\documentclass[11pt,reqno]{amsart}

\usepackage[margin=1in]{geometry}
\usepackage[table]{xcolor}
\usepackage{amssymb}
\usepackage[hidelinks]{hyperref}
\usepackage[noabbrev,capitalize]{cleveref}
\usepackage{microtype}
\usepackage{mathtools}

\newtheorem{thm}{Theorem}

\newtheorem{lem}[thm]{Lemma}
\newtheorem{prop}[thm]{Proposition}

\Crefname{lem}{Lemma}{Lemmas}

\theoremstyle{definition}

\title{Globally optimizing small codes in real projective spaces}
\author{Dustin G. Mixon \and Hans Parshall}
\address{Department of Mathematics, The Ohio State University, Columbus, OH 43210, USA}
\email{mixon.23@osu.edu}
\address{Department of Mathematics, The Ohio State University, Columbus, OH 43210, USA}
\email{parshall.6@osu.edu}

\begin{document}

\begin{abstract}
For $d\in\{5,6\}$, we classify arrangements of $d + 2$ points in $\mathbf{RP}^{d-1}$ for which the minimum distance is as large as possible.
To do so, we leverage ideas from matrix and convex analysis to determine the best possible codes that contain equiangular lines, and we introduce a notion of approximate Positivstellensatz certificates that promotes numerical approximations of Stengle's Positivstellensatz certificates to honest certificates.
\end{abstract}

\maketitle

\section{Introduction}

Given a compact metric space $(X,d)$ and a positive integer $n$, it is natural to consider a subset $C\subseteq X$ that maximizes the minimum distance $\delta(C):=\min_{x,y\in C,x\neq y}d(x,y)$.
Such a subset, known as an \textbf{optimal $n$-code} for $(X,d)$, is guaranteed to exist by compactness.
Optimal codes are maximally robust to noise, since one can identify $x\in C$ from any noisy version $\hat{x}\in X$ that satisfies $d(\hat{x},x)<\delta(C)/2$.
Optimal codes have been an object of study ever since a legendary dispute in 1694 between Isaac Newton and David Gregory~\cite{casselman04}.
Consider the unit sphere $S^2\subseteq \mathbf{R}^3$ with distance inherited from the ambient Euclidean distance.
In our language, Gregory asserted that an optimal $13$-code $C$ for $S^2$ has $\delta(C)\geq1$.
Interest in spherical codes was rejuvenated in 1930 by the Dutch botanist Tammes, who studied the distribution of pores on pollen grains~\cite{tammes30}.
Thanks to this resurgence, Gregory was finally proved wrong in 1953 by Sch\"{u}tte and van der Waerden~\cite{schutte53}.

In 1948, Claude Shannon founded the field of information theory~\cite{shannon48}, which in turn motivated the pursuit of optimal codes over $\mathbf{Z}_2^n$ with Hamming distance.
Noteworthy optimal codes in this metric space include the Golay code~\cite{golay49} and the Hamming code~\cite{hamming50}.
This metric space can be viewed in terms of the Cayley graph on $\mathbf{Z}_2^n$ with generators given by the identity basis.
More generally, every graph produces a metric space consisting of the vertex set and the graph's geodesic distance.
In this language, the independence number of a graph is the largest $n$ for which an optimal $n$-code $C$ satisfies $\delta(C)>1$.
For example, the independence number of the Paley graph is of particular interest in number theory~\cite{chung89,hanson19}.
The connection between optimal codes and independence numbers has been rather fruitful, as the Lov\'{a}sz--Schrijver bound can be generalized to obtain useful bounds for a variety of metric spaces~\cite{delaat15}.

In 1996, Conway, Hardin and Sloane~\cite{conway96} posed the problem of finding optimal codes for Grassmannian spaces with \textbf{chordal distance}, defined as follows:
Given two subspaces $U,V\subseteq\mathbf{F}^d$ of dimension $r$ with principal angles $\{\theta_i\}_{i\in[r]}$, then $d(U,V)=(\sum_i\sin^2\theta_i)^{1/2}$.
In the time since this seminal paper, there has been a flurry of progress in the special case of projective spaces due in part to emerging applications in multiple description coding~\cite{strohmer03}, digital fingerprinting~\cite{mixon13}, compressed sensing~\cite{bandeira13}, and quantum state tomography~\cite{renes04}.
Most of this work takes a particular form:
Identify a collection $S$ of mathematical objects such that for every $s\in S$, there exists an explicit optimal $n$-code in $\mathbf{FP}^{d-1}$, where $n=n(s)$ and $d=d(s)$.
For example, one may take $S$ to be the set of regular two-graphs~\cite{seidel76}, $n(s)$ the number of vertices in $s$, and $d(s)$ the multiplicity of the positive eigenvalue of $s$; indeed, for every $s\in S$, one may construct an optimal $n$-code for $\mathbf{RP}^{d-1}$ known as an \textit{equiangular tight frame}~\cite{strohmer03}.
See \cite{fickus15} for a survey of these developments.

Due to this style of progress, the current literature on optimal codes for real projective spaces is rather spotty; while we have provably optimal $n$-codes for $\mathbf{RP}^{d-1}$ for infinitely many $(d,n)$, large gaps remain.
In what follows, we identify where these gaps first emerge.
It is straightforward to verify that for $n\leq d$, the optimal $n$-codes for $\mathbf{RP}^{d-1}$ correspond to orthogonal lines.
For $n=d+1$, the optimal $n$-codes are obtained from regular simplices centered at the origin; indeed, the lines spanned by the vertices correspond to an equiangular tight frame~\cite{strohmer03,fickus15}.
However, for $n=d+2$, the optimal $n$-codes for $\mathbf{RP}^{d-1}$ are unknown for most values of $d$.
In this paper, we focus on this minimal case.

Since $\mathbf{RP}^1$ is a circle, the $d=2$ case is uniquely solved by four uniformly spaced points.
The $d=3$ case is far less trivial, and was originally solved by Fejes T\'{o}th~\cite{fejes65} in 1965.
The optimal code is unique up to isometry, and can be obtained by removing any one of the six lines that are determined by antipodal vertices of the icosahedron.
A second treatment of this proof was provided by Benedetto and Kolesar~\cite{benedetto06} in 2006.
Finally, Fickus, Jasper and Mixon~\cite{fickus18} gave a third, more general treatment in 2018, and the ideas of their proof also solved the $d=4$ case.
This optimal code is unique up to isometry, and corresponds to the putatively optimal code provided by Sloane on his website~\cite{sloaneDatabase}.
Following~\cite{fickus18}, the putatively optimal codes for $d\in\{5,6\}$ can be expressed in terms of Gram matrices of unit-vector representatives of each line:
\begin{equation}\label{eq:winners}
\normalsize{G_5 := \scriptsize{\left[
\begin{matrix*}[r]
1 & -a & a & -a & \phantom{-}a & -a & a \\
-a & 1 & a & a & a & -a & a \\
a & a & 1 & -a & a & a & -a\\
-a & a & -a & 1 & a & -a & -a\\
a & a & a & a & 1 & a & a\\
-a & -a & a & -a & a & 1 & -a\\
a & a & -a & -a & a & -a & 1
\end{matrix*}\right]}, \quad
G_6 := \scriptsize{\left[
\begin{matrix*}[r]
1 & b & b & -b & b & c & b & -b\\
b & 1 & -b & -b & -b & -b & -c & -b\\
b & -b & 1 & -b & -b & -b & -b & -b\\
-b & -b & -b & 1 & b & -b & b & -b\\
b & -b & -b & b & 1 & -b & -b & b\\
c & -b & -b & -b & -b & 1 & b & -b\\
b & -c & -b & b & -b & b & 1 & b\\
-b & -b & -b & -b & b & -b & b & 1
\end{matrix*}\right]}}
\end{equation}
where $a>0$ is the second smallest root of $x^3-9x^2-x+1$, $b>0$ is the second smallest root of
\begin{equation}
\label{eq.6x8 coherence}
106x^6 - 264x^5 - 53x^4 + 84x^3 + 20x^2 - 4x - 1,
\end{equation}
and $c\in(0,b)$ is the fourth smallest root of
\[
53x^6+484x^5+814x^4-860x^3-347x^2+352x-32.
\]
Note that $\sqrt{1-a^2}$ and $\sqrt{1-b^2}$ are lower bounds on the minimum chordal distance of optimal $(d+2)$-codes for $\mathbf{RP}^{d-1}$ for $d\in\{5,6\}$, respectively.
Recently, Bukh and Cox~\cite{bukh19} proved the best-known upper bound on this minimum distance:
\begin{equation}
\label{eq.bukh--cox}
\delta(C)\leq\sqrt{1-\big(\tfrac{3}{2d+1}\big)^2},
\end{equation}
and furthermore, they characterized the codes that achieve equality in this bound, which occurs for every $d\equiv 1\bmod 3$.
In particular, this gives an alternate proof of the $d=4$ case.
As a result, the next open cases are $d\in\{5,6,8\}$.

In this paper, we resolve the cases of $d\in\{5,6\}$.
As conjectured, $G_5$ and $G_6$ above describe optimal codes for $\mathbf{RP}^4$ and $\mathbf{RP}^5$, which are unique up to isometry.
The next section reviews the preliminaries that set up our approach.
In particular, \cref{prop:fjm} (that is, Lemma~6 in~\cite{fickus18}) implies that every optimizer is necessarily an optimizer of one of a handful of subprograms.
These subprograms come in two different species, and in Section~3, we apply ideas from matrix and convex analysis to solve the first species; specifically, we determine the best possible $(d+2)$-codes that contain $d+1$ equiangular lines.
In Section~4, we apply this theory to the $d=5$ case, and we solve the second species with a clever application of cylindrical algebraic decomposition.
This approach does not scale to the $d=6$ case.
As an alternative, Section~5 introduces a method to convert numerical approximations of Stengle's Positivstellensatz certificates into honest certificates.
This allows us to tackle the $d=6$ case in Section~6, where we solve the second species of subprograms by computing numerical approximations of Positivstellensatz certificates using a Julia-based implementation of sum-of-squares programming.
We conclude in Section~7 by discussing opportunities for future work.

The proofs of our main results are computer assisted. Our computations were performed on a 3.4 GHz Intel Core i5, and we report runtimes throughout to help identify computational bottlenecks. While our code is far from optimized, we make it available with the arXiv version of this paper.

\section{Preliminaries}

We identify an $n$-code for $\mathbf{RP}^{d - 1}$ with a set of $n$ lines through the origin of $\mathbf{R}^d$.  We seek to classify the optimal $n$-codes, that is, sets of $n$ lines for which the minimum angle between any two lines is maximized.  The cosine of the minimum angle is known as the \textbf{coherence}, and so classifying optimal $n$-codes is equivalent to classifying sets of lines with the minimum coherence.

Let $U^{d \times n}$ denote the set of $d \times n$ real matrices with unit-norm columns.  We can specify an $n$-code for $\mathbf{RP}^{d - 1}$ using a matrix $\Phi \in U^{d \times n}$ whose column vectors span the $n$ lines in $\mathbf{R}^d$. Let $B_n$ denote the group of $n \times n$ signed permutation matrices.  Observe that $\Phi, \Psi \in U^{d \times n}$ specify the same $n$-code of (unordered) points in $\mathbf{RP}^{d - 1}$ if and only if there exists $P \in B_n$ such that $\Phi P^T = \Psi$.  Moreover, $\Phi, \Psi \in U^{d \times n}$ specify the same $n$-code for $\mathbf{RP}^{d - 1}$ \emph{up to isometry} if and only if there exists $P \in B_n$ such that $P \Phi^T \Phi P^T = \Psi^T \Psi$.  Let $E_{n,d}$ denote the rank-constrained elliptope
\[
E_{n,d} := \{G \in \mathbf{R}^{n \times n} : G = G^T, \operatorname{diag}(G) = \mathbf{1}, G \succeq 0, \operatorname{rank}(G) \leq d\}.
\]
We say that $G, G' \in E_{n,d}$ are \textbf{equivalent} if there exists $P \in B_n$ such that $P G P^T = G'$.  Observe that the resulting equivalence classes correspond to isometry classes of $n$-codes for $\mathbf{RP}^{d - 1}$, and we can recover a representation $\Phi \in U^{d \times n}$ of one such $n$-code by decomposing $G = \Phi^T \Phi$.  The coherence of the lines represented by $G$ is given by
\[
\mu(G) := \max_{1 \leq i < j \leq n} |G_{ij}|.
\]
Hence, our problem is equivalent to computing
\[
\mu_{n,d} := \inf\{\mu(G) : G \in E_{n,d}\}
\]
and classifying the corresponding optimizer(s), which necessarily exist by compactness.  We say $G \in E_{n,d}$ is \textbf{optimal} if $\mu(G) = \mu_{n,d}$. 

In principle, one may directly apply Tarski--Seidenberg~\cite{mishra93} to find optimal $G$, but in practice, quantifier elimination over the reals is slow. For example, cylindrical algebraic decomposition (CAD)~\cite{collins75} is known to have runtimes that are doubly exponential in the number of variables~\cite{davenport88}, which is already too slow for the values of $n$ that we are interested in. For this reason, we need to somehow reduce the problem size before passing to tools like CAD. To this end, in the special case where $n=d+2$, optimal $G \in E_{n, d}$ are known to satisfy certain (strong) combinatorial constraints:

\begin{prop}[Lemma~6 in~\cite{fickus18}]\label{prop:fjm}
Suppose $G \in E_{d + 2, d}$ is optimal and put $\mu = \mu(G)$.
Then
\[
G = I + \mu S + X,
\]
where $I$ is the identity matrix, the matrices $I$, $S$ and $X$ are symmetric with disjoint support, the entries of $X$ all reside in $[-\mu,\mu]$, and the entrywise absolute value $|S|$ is the adjacency matrix of either
(i) $K_{d+1}$ union an isolated vertex or (ii) the complement of a maximum matching.
\end{prop}

In what follows, we assume $n=d+2$ without mention.
\cref{prop:fjm} considerably reduces the search space for optimal $G\in E_{n,d}$.
Let $\mathcal{S}_1\subseteq\mathbf{R}^{n \times n}$ denote the set of symmetric $S$ for which $|S|$ is the adjacency matrix of $K_{d + 1}$ union an isolated vertex, and similarly, let $\mathcal{S}_2\subseteq\mathbf{R}^{n \times n}$ denote the set of symmetric $S$ for which $|S|$ is the adjacency matrix of the complement of a maximum matching.
Letting $\circ$ denote entrywise matrix product, then for each $S \in \mathcal{S}_1 \cup \mathcal{S}_2$, we consider the subprogram  
\[
m(S)
:= \inf\{ \mu : I + \mu S + X \in E_{n, d}, (I + S)\circ X=0, -\mu\leq X_{ij} \leq\mu \}.
\]
\cref{prop:fjm} implies that $\mu_{n,d}=\min\{m(S):S\in\mathcal{S}_1\cup\mathcal{S}_2\}$, and we can recover each optimal $G \in E_{n, d}$ from the minimizers of $m(S)$.
We call $S \in \mathcal{S}_1 \cup \mathcal{S}_2$ \textbf{optimal} if $m(S) = \mu_{n, d}$.

Given $P\in B_n$, then $(\mu,X)$ is feasible in the program defining $m(S)$ if and only if $(\mu,PXP^T)$ is feasible in the program defining $m(PSP^T)$.
We may leverage this symmetry to further simplify our search for optimal $S$.
In particular, for each $i\in\{1,2\}$, the conjugation action of $B_n$ partitions $\mathcal{S}_i$ into orbits, and we say that two members of the same orbit are \textbf{equivalent}.
We may select a representative from each orbit to produce $\mathcal{R}_i\subseteq\mathcal{S}_i$.
Then
\[
\mu_{n,d}=\min\{m(S):S\in\mathcal{R}_1\cup\mathcal{R}_2\},
\]
and furthermore, every optimal $G\in E_{n,d}$ corresponding to an optimal $S\not\in\mathcal{R}_1\cup\mathcal{R}_2$ is equivalent to some optimal $G'\in E_{n,d}$ corresponding to an optimal $S'\in\mathcal{R}_1\cup\mathcal{R}_2$.  We select the members of $\mathcal{R}_1$ to be zero in the last row and column and the members of $\mathcal{R}_2$ to be zero in the last $\lfloor n/2 \rfloor$ diagonal $2\times 2$ blocks. This determines the support of both types of matrices.

\begin{table}
\begin{tabular}{cccrrrcl}
$d$ & $n$ & $\mu$ & \qquad min polynomial & \qquad $|\mathcal{R}_1|$ & \qquad $|\mathcal{R}_2|$ & \qquad & optimality\\
\hline
$3$ & $5$ & $0.4473$ & $5x^2 - 1$ & $3$ & $6$ && Ref.~\cite{fejes65,benedetto06,fickus18}\\
$4$ & $6$ & $0.3334$ & $3x - 1$ & $7$ & $14$ && Ref.~\cite{fickus18,bukh19}\\
$5$ & $7$ & $0.2863$ & $x^3 - 9x^2 - x + 1$ & $16$ & $144$ && Thm.~\ref{thm:5by7}\\
$6$ & $8$ & $0.2410$ & Eq.~\eqref{eq.6x8 coherence} & $54$ & $560$ && Thm.~\ref{thm:6by8}\\
$7$ & $9$ & $0.2000$ & $5x - 1$ & $243$ & $49,127$ && Ref.~\cite{bukh19}\\
$8$ & $10$ & $0.1828$ & $19x^2 + 2x - 1$ & $2,038$ & $599,108$ && ---
\end{tabular}
\caption{Parameters of optimal $n$-codes for $\mathbf{RP}^{d-1}$ with $n=d+2$. Coherence $\mu$ is rounded to the next multiple of $10^{-4}$, and we provide the minimal polynomial of $\mu$ to specify its precise value. As a consequence of \cref{prop:fjm}, every optimizer is necessarily an optimizer of a subprogram specified by some $S\in\mathcal{R}_1\cup\mathcal{R}_2$, suggesting that one solves each of these subprograms; the sheer number of subprograms makes this approach infeasible for the $d=8$ case.
For the other cases, we provide the location(s) of the proof(s) of optimality.
\label{table.opt packings}}
\end{table}

As we will see, optimizing over $\mathcal{R}_1$ is easier than optimizing over $\mathcal{R}_2$, and we will apply different techniques to perform these optimizations.
Before discussing these techniques, we first determine the sizes of $\mathcal{R}_1$ and $\mathcal{R}_2$ to help establish which values of $d$ are amenable to this approach.
For every member of $\mathcal{R}_1$, the off-diagonal entries are only nonzero on the leading $(d+1)\times(d+1)$ principal submatrix.
Restricting to this submatrix, then the members of $\mathcal{R}_1$ are precisely the Seidel adjacency matrices of switching class representatives on $d+1$ vertices, which were counted by Mallows and Sloane~\cite{mallows75}.
In \cref{table.opt packings}, we report the size of $\mathcal{R}_1$ for $d\in\{3,\ldots,8\}$.

The size of $\mathcal{R}_2$ does not appear in the literature, and so we apply Burnside's lemma to formulate a fast algorithm that computes it.
Let $\mathcal{T}_2\supseteq\mathcal{R}_2$ denote the subset of $\mathcal{S}_2$ that is zero in the last $\lfloor n/2 \rfloor$ diagonal $2\times 2$ blocks. Next, consider the conjugation action of $B_n$ on $\mathcal{S}_2$, and let $F_2$ denote the largest subgroup of $B_n$ that acts invariantly on $\mathcal{T}_2$. Then our choice for $\mathcal{R}_2$ equates to representatives of orbits of the action of $F_2$ on $\mathcal{T}_2$. By Burnside, the size of $\mathcal{R}_2$ then equals the average number of points in $\mathcal{T}_2$ that are fixed by a random member of $F_2$. By construction, every member of $\mathcal{T}_2$ has the same support above the diagonal $E\subseteq\{(i,j):1\leq i<j\leq n\}$, and for each $P\in F_2$, the mapping $X\mapsto PXP^T$ over symmetric $X$ induces a signed permutation $P_E$ over $\mathbf{R}^E$, which enjoys a unique decomposition into disjoint signed cycles. If any of these cycles features an odd number of sign changes, then there is no $x\in\{\pm1\}^E$ for which $P_Ex=x$, and so $P$ has no fixed points in $\mathcal{T}_2$. Write $O\subseteq F_2$ for this subset of $P$'s. If $P\not\in O$, then the number of points fixed by $P$ equals $2^{k(P)}$, where $k(P)$ denotes the number of disjoint signed cycles in the decomposition of $P_E$. Overall, we have
\[
|\mathcal{R}_2|
=\frac{1}{|F_2|}\sum_{P\in F_2}
\left\{\begin{array}{cl}
0&\text{if }P\in O\\
2^{k(P)}&\text{else}
\end{array}\right\},
\]
which can be computed quickly by iterating over members of $B_n$. See \cref{table.opt packings} for the result of this computation for $d\in\{3,\ldots,8\}$.

In what follows, we describe our methodology for minimizing $m(S)$ subject to $S\in\mathcal{R}_1\cup\mathcal{R}_2$ in the cases where $d\in\{5,6\}$. In vague terms, our approach performs a computation for each $S$ and then compares the results. Considering \cref{table.opt packings}, we expect this approach to require about a thousand times as much runtime to resolve the next open case of $d=8$, even if the per-$S$ runtime matches the $d=6$ case (in reality, it is slower). As such, new ideas will be necessary to tackle this case.

\section{Codes from equiangular lines}

In this section, we prove results that will help us to estimate $m(S)$ for every $S\in\mathcal{R}_1$.

\begin{lem}\label{lem:eig}
Let $S \in \mathcal{R}_1$ and let $\lambda$ be the minimum eigenvalue of its leading $(d + 1) \times (d + 1)$ principal submatrix.  Then $m(S) \in \{-\lambda^{-1}, \infty\}$.
\end{lem}

\begin{proof}
Suppose $m(S) \neq \infty$.  Then there exist $\mu$ and $X$ such that
\[
I + \mu S + X \in E_{d + 2, d},
\qquad
(I + S)\circ X=0,
\qquad
-\mu\leq X_{ij} \leq\mu.
\]
In particular, $I + \mu S + X \succeq 0$, and so $I + \mu S' \succeq 0$, where $S'$ is the leading $(d + 1) \times (d + 1)$ principal submatrix of $S$.  Furthermore, $I + \mu S'$ has rank at most $d$, and so $1 + \mu \lambda = 0$.
\end{proof}

The next result requires a definition:
We say $\{v_i\}_{i\in[l]}$ in $\mathbf{R}^d$ are conically dependent if there exists $j\in[l]$ and nonnegative $\{\alpha_i\}_{i\in[l]\setminus\{j\}}$ such that
\[
v_j
=\sum_{i\in[l]\setminus\{j\}}\alpha_iv_i.
\]
Otherwise, we say $\{v_i\}_{i\in[l]}$ are \textbf{conically independent}.

\begin{lem}\label{lem:eiginf}
Let $S \in \mathcal{R}_1$, suppose the minimum eigenvalue $\lambda < 0$ of its leading $(d + 1) \times (d + 1)$ principal submatrix $S'$ has multiplicity $1$, take $L\in\mathbf{R}^{(d+1)\times d}$ such that $LL^T = I - \lambda^{-1}S'$, and consider the pseudoinverse given by $L^\dagger=(L^TL)^{-1}L^T$.
\begin{itemize}
\item[(a)]
Suppose $\| L^\dagger y \|_2 < -\lambda$ for every $y \in \{\pm 1\}^{d + 1}$.
Then $m(S) = \infty$.
\item[(b)]
Suppose there exists a nonempty subset $\mathcal{Y}\subseteq \{\pm 1\}^{d + 1}$ such that $\{L^\dagger y\}_{y\in\mathcal{Y}}$ is conically independent, $\| L^\dagger y \|_2 < -\lambda$ for every $y \in \{\pm 1\}^{d + 1}\setminus\mathcal{Y}$, and for every $y\in\mathcal{Y}$, the matrix
\[
Z(y):=\begin{bmatrix}0 & y\\ y^T & 0\end{bmatrix}
\]
has the property that $S+Z(y)$ has minimum eigenvalue $\lambda$ with multiplicity $2$.
Then $m(S) = -\lambda^{-1}$ and the corresponding minimizers are given by $X=-\lambda^{-1} Z(y)$ for $y\in\mathcal{Y}$.
\end{itemize}
\end{lem}

\begin{proof}
First, $\lambda < 0$ since $S$ is a nonzero matrix with zero trace.
Hence, $I - \lambda^{-1}S' \succeq 0$, and since $I - \lambda^{-1}S'$ has rank at most $d$, there exists $L\in\mathbf{R}^{(d+1)\times d}$ such that $LL^T = I - \lambda^{-1}S'$.
In fact, $L$ has rank exactly $d$ since $\lambda$ is an eigenvalue of $S'$ with multiplicity $1$.

(a)
We will prove this claim by contraposition, and so we suppose $m(S) \neq \infty$. 
By \cref{lem:eig}, it follows that $m(S) = -\lambda^{-1}$.
Set $\mu = m(S)$ and consider the set
\[
\mathcal{X} := \{X : I + \mu S + X \in E_{d + 2, d}, (I + S)\circ X=0, -\mu\leq X_{ij} \leq\mu\}.
\]
Since $m(S) \neq \infty$, a compactness argument gives that $\mathcal{X}$ is nonempty, and we may select $X \in \mathcal{X}$ and obtain a decomposition of the form $I + \mu S + X = A^TA$, where $A = [L^T~x]$ and $x \in \mathbf{R}^d$ is a unit vector satisfying $\| Lx \|_\infty \leq \mu = -\lambda^{-1}$.
Since $L$ has rank $d$, it holds that $\| Lx \|_\infty>0$.
Thus,
\begin{align*}
-\lambda
\leq \| Lx \|_\infty^{-1}
\leq \sup_{\|z\|_2=1} \| Lz \|_\infty^{-1}
= \sup_{z \neq 0} \frac{ \| z \|_2}{ \| Lz \|_\infty }
&= \sup_{y \in \operatorname{im}(L)\setminus \{0\}} \frac{\| L^\dagger y \|_2}{\| y \|_\infty}\\
&\leq \sup_{y \neq 0} \frac{\| L^\dagger y \|_2}{\| y \|_\infty}
= \sup_{y \in B_\infty^{d + 1}} \| L^\dagger y \|_2
=\max_{y \in \{\pm 1\}^{d + 1}} \| L^\dagger y \|_2,
\end{align*}
where the last step uses the fact that the maximum of a convex function over a compact polytope is achieved at a vertex of that polytope.

(b)
Since $\mathcal{Y}$ is nonempty, there exists $y$ such that $I-\lambda^{-1}(S+Z(y))$ is positive semidefinite with rank $d$, and so $m(S)\neq\infty$.
Then by \cref{lem:eig}, it holds that $m(S) = -\lambda^{-1}$.
It remains to show that the minimizers $X\in \mathcal{X}$ of the program defining $m(S)$ are $X=-\lambda^{-1}Z(y)$ for $y\in\mathcal{Y}$.

First, we show that $\| L^\dagger y \|_2=-\lambda$ for every $y\in\mathcal{Y}$.
To see this, fix $y\in\mathcal{Y}$ and consider the decomposition $I-\lambda^{-1}(S+Z(y))=A^TA$, where $A=[L^T~x]$.
Then $x$ has unit norm and $Lx=-\lambda^{-1}y$.
We apply $-\lambda L^\dagger$ to both sides and take norms to get $\| L^\dagger y \|_2=\|-\lambda x\|_2=-\lambda$.
As such,
\begin{equation}
\label{eq.max minus lambda}
\max_{y \in \{\pm 1\}^{d + 1}} \| L^\dagger y \|_2
=-\lambda.
\end{equation}
Next, we follow the proof of (a) to see that every $X\in \mathcal{X}$ yields a decomposition $I+\mu S+X=A^TA$ with $A=[L^T~x]$, where $x$ has unit norm and
\[
-\lambda
\stackrel{(*)}{\leq} \| Lx \|_\infty^{-1}
\stackrel{(\dagger)}{\leq} \sup_{\|z\|_2=1} \| Lz \|_\infty^{-1}
= \sup_{y \in \operatorname{im}(L)\setminus \{0\}} \frac{\| L^\dagger y \|_2}{\| y \|_\infty}
\stackrel{(\ddagger)}{\leq} \sup_{y \neq 0} \frac{\| L^\dagger y \|_2}{\| y \|_\infty}
=\max_{y \in \{\pm 1\}^{d + 1}} \| L^\dagger y \|_2
=-\lambda,
\]
where the last step comes from \eqref{eq.max minus lambda}. 
By equality, we may conclude a few things.
First, equality in ($\dagger$) implies $x\in\arg\max\{\|Lz\|_\infty^{-1}:\|z\|_2=1\}$, and so a change of variables gives
\[
Lx
\in\arg\max\{\|y\|_\infty^{-1}:\|L^\dagger y\|_2=1,y\in\operatorname{im}(L)\}
\subseteq\arg\max\Big\{\tfrac{\|L^\dagger y\|_2}{\|y\|_\infty}:y\in\operatorname{im}(L)\setminus\{0\}\Big\}.
\]
Next, equality in ($*$) implies $\|Lx\|_\infty=-\lambda^{-1}$, and so we further have
\[
-\lambda Lx
\in\arg\max\{\|L^\dagger y\|_2:\|y\|_\infty\leq1,y\in\operatorname{im}(L)\}
\subseteq\arg\max\{\|L^\dagger y\|_2:\|y\|_\infty\leq1\},
\]
where the last step follows from equality in ($\ddagger$).
We claim that $\arg\max\{\|L^\dagger y\|_2:\|y\|_\infty\leq1\}=\mathcal{Y}$.
Our result follows from this intermediate claim since $-\lambda Lx=y\in\mathcal{Y}$ implies
\[
I+\mu S+X
=A^TA
=\begin{bmatrix}LL^T & Lx\\ x^TL^T & x^Tx\end{bmatrix}
=\begin{bmatrix}I-\lambda^{-1}S' & -\lambda^{-1} y\\ -\lambda^{-1} y^T & 1\end{bmatrix}
=I+\mu S- \lambda^{-1}Z(y),
\]
and so rearranging gives that every minimizer $X\in\mathcal{X}$ is of the form $X=-\lambda^{-1}Z(y)$, as desired.

We use convexity to prove $\mathcal{M}:=\arg\max\{\|L^\dagger y\|_2:\|y\|_\infty\leq1\}=\mathcal{Y}$.
First, we know $\mathcal{Y}\subseteq \mathcal{M}$ since the maximum of a convex function over a compact polytope is achieved at a vertex of that polytope.
For the sake of contradiction, suppose this containment is proper, that is, there exists $y_0\in \mathcal{M}\setminus\mathcal{Y}$.
By convexity, we may write $y_0=\sum_{v\in\{\pm1\}^{d+1}} c_v v$ with $c_v\geq0$ and $\sum_vc_v=1$.

In what follows, we show that $c_v>0$ for some $v\in\{\pm1\}^{d+1}\setminus\mathcal{Y}$.
Suppose otherwise that $c_v$ is only nonzero for $v\in\mathcal{Y}$.
Since $y_0\not\in\mathcal{Y}$, then there exists a subset $\mathcal{Y}'\subseteq\mathcal{Y}$ of size at least $2$ such that $c_v$ is nonzero precisely when $v\in\mathcal{Y}'$.
By assumption, $\{L^\dagger y\}_{y\in\mathcal{Y}'}$ is conically independent.
As such, picking $y_1\in\mathcal{Y}'$, it holds that $c_{y_1}L^\dagger y_1$ is not a positive scalar multiple of $\sum_{y\in\mathcal{Y}'\setminus\{y_1\}}c_y L^\dagger y$, and so
\begin{align*}
-\lambda
=\|L^\dagger y_0\|_2
=\bigg\|L^\dagger \sum_{y\in\mathcal{Y}'}c_y y\bigg\|_2
&=\bigg\|c_{y_1} L^\dagger y_1+\sum_{y\in\mathcal{Y}'\setminus\{y_1\}}c_y L^\dagger y\bigg\|_2\\
&<\|c_{y_1} L^\dagger y_1\|_2+\bigg\|\sum_{y\in\mathcal{Y}'\setminus\{y_1\}}c_y L^\dagger y\bigg\|_2
\leq \sum_{y\in\mathcal{Y}'}c_y\|L^\dagger y\|_2
=-\lambda,
\end{align*}
a contradiction.
Overall, it must be the case that $c_v>0$ for some $v\in\{\pm1\}^{d+1}\setminus\mathcal{Y}$.

Finally, $\|L^\dagger y_0\|_2=\|L^\dagger y\|_2=-\lambda$ for every $y\in\mathcal{Y}$ and $\|L^\dagger y_0\|_2\leq\sum_{v\in\{\pm1\}^{d+1}} c_v \|L^\dagger v\|_2$, and so
\begin{align*}
-\lambda
&=\frac{1}{1-\sum_{y\in\mathcal{Y}}c_y}\Big(\|L^\dagger y_0\|_2-\sum_{y\in\mathcal{Y}}c_y\|L^\dagger y\|_2\Big)\\
&\leq\frac{1}{1-\sum_{y\in\mathcal{Y}}c_y}\sum_{v\in\{\pm1\}^{d+1}\setminus\mathcal{Y}} c_v \|L^\dagger v\|_2
\leq \max_{v\in\{\pm1\}^{d+1}\setminus\mathcal{Y}}\|L^\dagger v\|_2
\leq -\lambda.
\end{align*}
By equality, we then conclude that $\max_{v\in\{\pm1\}^{d+1}\setminus\mathcal{Y}}\|L^\dagger v\|_2=-\lambda=\|L^\dagger y\|_2$ for every $y\in\mathcal{Y}$, which contradicts the fact that $\arg\max\{\|L^\dagger y\|_2:y\in\{\pm1\}^{d+1}\}=\mathcal{Y}$.
\end{proof}

\section{The optimal \texorpdfstring{$7$-code}{7-code} for \texorpdfstring{$\mathbf{RP}^4$}{RP4}}\label{5by7section}

In this section we fix $n = 7$ and $d = 5$ and prove the following classification.

\begin{thm}\label{thm:5by7}
$G \in E_{7,5}$ is optimal if and only if $G$ is equivalent to $G_5$, given in \eqref{eq:winners}.
\end{thm}

\begin{proof}
First, we recall the bounds on $\mu_{7,5}$ implied by the Bukh--Cox bound in \eqref{eq.bukh--cox} and the code represented by $G_5$ in \eqref{eq:winners}:
\begin{equation}\label{eq:5by7bounds}
\frac{3}{11} \leq \mu_{7,5} \leq 0.2863.
\end{equation}
These bounds will play a role in our analysis of both $\mathcal{R}_1$ and $\mathcal{R}_2$.

Let $\mathcal{T}_1$ denote the subset of $\mathcal{S}_1$ that is zero in its last row and column, and let $F_1$ be the subgroup of $B_n$ that acts invariantly on $\mathcal{T}_1$.  Every member of $\mathcal{T}_1$ is equivalent to a matrix of the form
\begin{equation}\label{eq:r1n7}
\normalsize{
\scriptsize{\left[
\begin{array}{rrrrrrr}
0 & 1 & 1 & 1 & 1 & 1 & \phantom{\pm} 0\\
1 & 0 & \pm 1 & \pm 1 & \pm 1 & \pm 1 & 0\\
1 & \pm 1 & 0 & \pm 1 & \pm 1 & \pm 1 & 0\\ 
1 & \pm 1 & \pm 1 & 0 & \pm 1 & \pm 1 & 0\\
1 & \pm 1 & \pm 1 & \pm 1 & 0 & \pm 1 & 0\\
1 & \pm 1 & \pm 1 & \pm 1 & \pm 1 & 0 & 0\\
0 & 0 & 0 & 0 & 0 & 0 & 0
\end{array}
\right]}
}
\end{equation}
and so we can generate orbits of $\mathcal{T}_1$ under the action of $F_1$ by generating the orbits of these $2^{10}$ matrices.  We then build $\mathcal{R}_1$ by selecting one representative from each orbit of the form \eqref{eq:r1n7}, and this takes under one second.

For each $S \in \mathcal{R}_1$, we compute the minimum eigenvalue $\lambda$ of its leading $6 \times 6$ principal submatrix.  By \cref{lem:eig} and \eqref{eq:5by7bounds}, we know that if $S$ is optimal, then $\frac{3}{11} \leq -\lambda^{-1} \leq 0.2863$, and this rules out all but two members of $\mathcal{R}_1$ from being optimal.  For each of these two remaining members, we verify that $\lambda$ has multiplicity 1, compute $L^\dagger$ according to the setup of \cref{lem:eiginf}, and compute $\|L^\dagger y\|_2$ for every $y \in \{\pm 1\}^6$.  In one case, we verify that $\|L^\dagger y\|_2 < -\lambda$ for every $y \in \{\pm 1\}^6$, and so this case is eliminated by \cref{lem:eiginf}(a).  For the only remaining $S \in \mathcal{R}_1$, we obtain $\mathcal{Y} \subseteq \{\pm 1\}^6$ satisfying the hypotheses of \cref{lem:eiginf}(b) and set $\mu = -\lambda^{-1} \approx 0.2863$ with minimal polynomial $x^3 - 9x^2 - x + 1$.  Applying \cref{lem:eiginf}(b) reveals that any minimizer $X$ for $m(S)$ leads to a Gram matrix $I + \mu S + X$ whose off-diagonal entries are all $\pm \mu$.  This corresponds to a set of 7 equiangular lines in $\mathbf{R}^5$, unique up to isometry, reported by Bussemaker and Seidel as the complement of the 25th two-graph of order 7 in Table 1 of~\cite{bussemaker81}.  To show that this configuration is optimal, we must still analyze $\mathcal{R}_2$.

Let $\mathcal{T}_2$ denote the subset of $\mathcal{S}_2$ with zero entries in its last 3 diagonal $2 \times 2$ blocks, and let $F_2$ be the subgroup of $B_n$ that acts invariantly on $\mathcal{T}_2$.  Every member of $\mathcal{T}_2$ is equivalent to a matrix of the form
\begin{equation}\label{eq:r2n7}
\normalsize{\scriptsize{
\left[
\begin{array}{rrrrrrr}
0 & 1 & 1 & 1 & 1 & 1 & 1\\
1 & 0 & 0 & \pm 1 & \pm 1 & \pm 1 & \pm 1\\
1 & 0 & 0 & \pm 1 & \pm 1 & \pm 1 & \pm 1\\
1 & \pm 1 & \pm 1 & 0 & 0 & \pm 1 & \pm 1\\
1 & \pm 1 & \pm 1 & 0 & 0 & \pm 1 & \pm 1\\
1 & \pm 1 & \pm 1 & \pm 1 & \pm 1 & 0 & 0\\
1 & \pm 1 & \pm 1 & \pm 1 & \pm 1 & 0 & 0\\
\end{array}
\right]},}
\end{equation}
and so we can generate orbits of $\mathcal{T}_2$ under the action of $F_2$ by generating the orbits of these $2^{12}$ matrices.  We then build $\mathcal{R}_2$ by selecting one representative from each orbit of the form~\eqref{eq:r2n7}, and this takes under one minute.

For each member $S \in \mathcal{R}_2$, we build the corresponding Gram matrix $G = I + \mu S + X$ with variable entries $\{\mu, X_{23}, X_{45}, X_{67}\}$.  We restrict $\mu$ according to \eqref{eq:5by7bounds} and $X_{23}, X_{45}, X_{67} \in [-\mu,\mu]$.  If $S$ is optimal, then there must be a choice of $\mu$ and $X$ for which $G$ is positive semidefinite and of rank $5$.  We can determine if such a choice of variables exists by solving the system of polynomial equalities and inequalities resulting from ensuring that each $6 \times 6$ minor of $G$ vanishes, some $5 \times 5$ minor of $G$ does not vanish, and each principal minor is nonnegative.  In principle, a solution is provided by CAD, but even after our reduction to this $4$-variable system, its exceedingly slow runtime makes it necessary to relax our problem.  We relax our rank and positive semidefinite constraints to simply ask for three $6 \times 6$ minors of $G$ to vanish, two of which are polynomials only in $\mu, X_{45}, X_{67}$, and the third of which is linear in $X_{23}$.  Then after roughly two minutes, CAD reports that out of the 144 representatives $S \in \mathcal{R}_2$, only 11 allow the prescribed minors to vanish with $\mu$ satisfying \eqref{eq:5by7bounds} and $X_{23}, X_{45}, X_{67} \in [-\mu,\mu]$.  Moreover, for each of these 11 representatives, $\mu$ is the root of $x^3 - 9x^2 - x + 1$ reported in Table 1 and $X_{23}, X_{45}, X_{67} \in \{\pm \mu\}$, and so each resulting Gram matrix corresponds to a set of equiangular lines with coherence $\mu$.  Each of these Gram matrices has rank 5, and therefore correspond to the previously described set of 7 equiangular lines in $\mathbf{R}^5$.
\end{proof}

Our use of CAD here does not scale to the $d=6$ case, and so the next section describes an alternative approach involving Stengle's Positivstellensatz.

\section{Approximate Positivstellensatz}

Let $\mathbf{R}[x]=\mathbf{R}[x_1,\ldots,x_n]$ denote the set of polynomials with real coefficients and variables $x_1,\ldots,x_n$.
Let $\Sigma^2[x]$ denote the set of polynomials that can be expressed as a sum of squares of polynomials from $\mathbf{R}[x]$.
Given $f_1(x),\ldots,f_k(x),g_1(x),\ldots,g_l(x)\in\mathbf{R}[x]$, put $f:=\{f_i(x)\}_{i\in[k]}$ and $g:=\{g_j(x)\}_{j\in[l]}$, and consider the sets
\[
P(f):=\Big\{x\in\mathbf{R}^n:f_i(x)\geq0~\forall i\in[k]\Big\},
\qquad
Z(g):=\Big\{x\in\mathbf{R}^n:g_j(x)=0~\forall j\in[l]\Big\}.
\]
Then every polynomial in the cone
\[
C(f)
:=\bigg\{\sum_{I\subseteq[k]}s_I(x)\prod_{i\in I}f_i(x):s_I(x)\in\Sigma^2[x] ~ \forall I\subseteq[k]\bigg\}
\]
is nonnegative over $P(f)$, while every polynomial in the ideal
\[
I(g)
:=\bigg\{\sum_{j\in[l]}t_j(x)g_j(x):t_j(x)\in\mathbf{R}[x] ~ \forall j\in[l]\bigg\}
\]
is zero over $Z(g)$.
As such, writing $p(x)+q(x)=-1$ with $p(x)\in C(f)$ and $q(x)\in I(g)$ would certify that $P(f)\cap Z(g)$ is empty.
Amazingly, such a certificate is available whenever $P(f)\cap Z(g)$ is empty:

\begin{prop}[Stengle's Positivstellensatz~\cite{stengle74}]
The following are equivalent:
\begin{itemize}
\item[(a)]
$P(f)\cap Z(g)=\emptyset$.
\item[(b)]
$-1\in C(f)+ I(g)$.
\end{itemize}
\end{prop}

In principle, one may hunt for Positivstellensatz certificates by fixing $D\in\mathbf{N}$ and restricting to a search for $p(x)$ and $q(x)$ of degree at most $D$, as this reduces to a semidefinite program.
As a proof of concept, Parrilo and Sturmfels~\cite{parrilo03} applied this method to prove that
\begin{equation}
\label{eq.ps example}
S=\Big\{(x,y)\in\mathbf{R}^2:x-y^2+3\geq0,~y+x^2+2=0\Big\}
\end{equation}
is empty.
In reproducing this proof, we found the Julia implementation of sum-of-squares programming to be particularly user-friendly~\cite{bezanson17,dunning17}.
However, as an artifact of numerical optimization, the resulting degree-$4$ polynomials $p(x)$ and $q(x)$ have the property that $p(x)+q(x)+1$ is also a degree-$4$ polynomial, but all of its coefficients have absolute value smaller than $10^{-12}$.
Indeed, numerical optimization will generally fail to deliver an exact Positivstellensatz certificate, meaning we cannot directly apply Stengle's Positivstellensatz.
As an alternative, we introduce the notion of an \textbf{approximate Positivstellensatz certificate}, taking inspiration from the approximate dual certificates that arise in compressed sensing and matrix completion~\cite{gross11,foucart13}.

\begin{lem}[Approximate Positivstellensatz]\label{approxpsatz}
Suppose $r>0$ and $\|x\|_\infty< r$ for every $x\in P(f)\cap Z(g)$.
Then the following are equivalent:
\begin{itemize}
\item[(a)]
$P(f)\cap Z(g)=\emptyset$.
\item[(b)]
There exists $h(x)=\sum_\alpha c_\alpha x^\alpha\in 1+ C(f)+ I(g)$ such that
\begin{equation}
\label{eq.bound on coefficients}
\max_\alpha|c_\alpha|
\leq\Bigg[\sum_{k=0}^{\operatorname{deg}h(x)}\binom{n+k-1}{k}r^k\Bigg]^{-1}.
\end{equation}
\end{itemize}
\end{lem}

\begin{proof}
To see that (a) implies (b), set $h = 0$ and apply Stengle's Positivstellensatz.  Now suppose $h \in1+ C(f)+ I(g)$ satisfies~\eqref{eq.bound on coefficients}.  If $h = 0$, then (a) follows from Stengle's Positivstellensatz.  Otherwise $h \neq 0$.  If (a) fails, then there exists $x \in P(f)\cap Z(g)$, where $h$ satisfies
\[
1
\stackrel{(*)}{\leq}
h(x)
\leq|h(x)|
\stackrel{(\dagger)}{\leq} \max_\alpha|c_\alpha|\cdot \sum_\alpha |x^\alpha|
\stackrel{(\ddagger)}{<} \max_\alpha|c_\alpha|\cdot \sum_{k=0}^{\operatorname{deg}h(x)}\binom{n+k-1}{k}r^k
\stackrel{(\S)}{\leq} 1,
\]
a contradiction.
In particular, ($*$) follows from the fact that $h(x)\in1+ C(f)+ I(g)$, ($\dagger$) uses the triangle inequality, ($\ddagger$) applies our assumptions that $h\neq0$ and $\|x\|_\infty<r$ and the count of monomials of each degree $k$, and finally ($\S$) applies the bound \eqref{eq.bound on coefficients}.
\end{proof}

Returning to the example \eqref{eq.ps example}, one can easily prove a bound on $\max\{|x|,|y|\}$ for every $(x,y)\in S$.
For example, if $|x|\geq 3$, then
\[
2|x|
\geq|x+3|
\geq|y|^2
=|x^2+2|^2
\geq|x|^4,
\]
which implies $|x|\leq 2^{1/3}$, a contradiction.
As such, if $(x,y)\in S$, then it must hold that $|x|<3$, and therefore $|y|=|x|^2+2<11$.
Now that we know that $\max\{|x|,|y|\}<r:=11$ for every $(x,y)\in S$, we recall that our numerical optimizer produced a degree-$4$ polynomial $h(x,y)\in 1+ C(f)+ I(g)$ whose coefficients $c_\alpha$ all have absolute value smaller than $10^{-12}$.
A short computation shows
\[
\max_\alpha|c_\alpha|
\leq 10^{-12}
\leq \Bigg[\sum_{k=0}^{4}\binom{n+k-1}{k}r^k\Bigg]^{-1},
\]
meaning $h(x,y)$ serves as an approximate Positivstellensatz certificate that $S=\emptyset$.

When $r<1$, we note that \eqref{eq.bound on coefficients} can be replaced by the simpler bound
\begin{equation}
\label{eq.simpler bound}
\max_\alpha|c_\alpha|
\leq (1-r)^n,
\end{equation}
since in this case it holds that
\[
\sum_{k=0}^{\operatorname{deg}h(x)}\binom{n+k-1}{k}r^k
\leq \sum_{k=0}^{\infty}\binom{n+k-1}{k}r^k
=(1-r)^{-n}.
\]
We will apply this simpler bound in our classification of optimal $8$-codes for $\mathbf{RP}^5$.

\section{The optimal \texorpdfstring{$8$-code}{8-code} for \texorpdfstring{$\mathbf{RP}^5$}{RP5}}

In this section we fix $n = 8$ and $d = 6$ and prove the following classification.

\begin{thm}\label{thm:6by8}
$G \in E_{8,6}$ is optimal if and only if $G$ is equivalent to $G_6$, given in \eqref{eq:winners}.
\end{thm}

\begin{proof}
First, we recall the bounds on $\mu_{8,6}$ implied by the Bukh--Cox bound in \eqref{eq.bukh--cox} and the code represented by $G_6$ in \eqref{eq:winners}:
\begin{equation}\label{eq:6by8bounds}
\frac{3}{13} \leq \mu_{8,6} \leq 0.2410.
\end{equation}
These bounds will play a role in our analysis of both $\mathcal{R}_1$ and $\mathcal{R}_2$.

Let $\mathcal{T}_1$ denote the subset of $\mathcal{S}_1$ that is zero in its last row and column, and let $F_1$ be the subgroup of $B_n$ that acts invariantly on $\mathcal{T}_1$.  Every member of $\mathcal{T}_1$ is equivalent to a matrix of the form
\begin{equation}\label{eq:r1n8}
\normalsize{\scriptsize{\left[
\begin{array}{rrrrrrrr}
0 & 1 & 1 & 1 & 1 & 1 & 1 & 0\\
1 & 0 & \pm 1 & \pm 1 & \pm 1 & \pm 1 & \pm 1 & \phantom{\pm} 0\\
1 & \pm 1 & 0 & \pm 1 & \pm 1 & \pm 1 & \pm 1 & 0\\ 
1 & \pm 1 & \pm 1 & 0 & \pm 1 & \pm 1 & \pm 1 & 0\\
1 & \pm 1 & \pm 1 & \pm 1 & 0 & \pm 1 & \pm 1 & 0\\
1 & \pm 1 & \pm 1 & \pm 1 & \pm 1 & 0 & \pm 1 & 0\\
1 & \pm 1 & \pm 1 & \pm 1 & \pm 1 & \pm 1 & 0 & 0\\
0 & 0 & 0 & 0 & 0 & 0 & 0 & 0
\end{array}
\right]}}
\end{equation}
and so we can generate orbits of $\mathcal{T}_1$ under the action of $F_1$ by generating the orbits of these $2^{15}$ matrices.  We then build $\mathcal{R}_1$ by selecting one representative from each orbit of the form \eqref{eq:r1n8}, and this takes under one minute.

For each $S \in \mathcal{R}_1$, we compute the minimum eigenvalue $\lambda$ of its leading $7 \times 7$ principal submatrix.  By \cref{lem:eig} and \eqref{eq:6by8bounds}, we know that if $S$ is optimal, then $\frac{3}{13} \leq -\lambda^{-1} \leq 0.2410$, and this rules out all but two members of $\mathcal{R}_1$ from being optimal.  Both of these are then ruled out by \cref{lem:eiginf}(a).  Thus, no member of $\mathcal{R}_1$ is optimal, and so we proceed to investigate $\mathcal{S}_2$.

Let $\mathcal{T}_2$ denote the subset of $\mathcal{S}_2$ with zero entries in its diagonal $2 \times 2$ blocks, and let $F_2$ denote the subgroup of $B_n$ that acts invariantly on $\mathcal{T}_2$.  Every member of $\mathcal{T}_2$ is equivalent to a matrix of the form
\begin{equation}\label{eq:r2n8}
\normalsize{\scriptsize{\left[
\begin{array}{rrrrrrrr}
0 & 0 & 1 & 1 & 1 & 1 & 1 & 1\\
0 & 0 & 1 & \pm 1 & \pm 1 & \pm 1 & \pm 1 & \pm 1\\
1 & 1 & 0 & 0 & \pm 1 & \pm 1 & \pm 1 & \pm 1\\
1 & \pm 1 & 0 & 0 & \pm 1 & \pm 1 & \pm 1 & \pm 1\\
1 & \pm 1 & \pm 1 & \pm 1 & 0 & 0 & \pm 1 & \pm 1\\
1 & \pm 1 & \pm 1 & \pm 1 & 0 & 0 & \pm 1 & \pm 1\\
1 & \pm 1 & \pm 1 & \pm 1 & \pm 1 & \pm 1 & 0 & 0\\
1 & \pm 1 & \pm 1 & \pm 1 & \pm 1 & \pm 1 & 0 & 0\\
\end{array}
\right]}}
\end{equation}
and so we can generate orbits of $\mathcal{T}_2$ under the action of $F_2$ by generating the orbits of these $2^{17}$ matrices.  We then build $\mathcal{R}_2$ by selecting one representative from each orbit of the form \eqref{eq:r2n8}; it takes roughly 15 minutes to produce the 560 elements of $\mathcal{R}_2$.

For 558 members of $\mathcal{R}_2$, we will show that they are not optimal by proving $m(S) > 0.2410$.  To do so, we introduce the decision variables $\{\mu, X_{12}, X_{34}, X_{56}, X_{78}\}$ and, for each member of $S \in \mathcal{R}_2$, build the symmetric matrix $G := I + \mu S + X$ with $(I + S) \circ X = 0$.  Then by definition, $m(S)$ is the infimum of $\mu$ such that $G$ is positive semidefinite with rank $6$ and $-\mu \leq X_{ij} \leq \mu$.  As in \cref{5by7section}, we will obtain a lower bound for $m(S)$ by completely relaxing the positive semidefinite constraint and partially relaxing the rank-6 constraint.  However, unlike that case, we were not able to find a suitable relaxation for which CAD both provided the necessary lower bound on $m(S)$ and also terminated in a reasonable amount of time.

We introduce the polynomials
\[
f := \{f_i\}_{i \in [10]} = \{\mu - \tfrac{3}{13}\} \cup \{0.2410 - \mu\} \cup \{\mu \pm X_{j,j + 1}\}_{j \in \{1,3,5,7\}},
\]
and we let $g = \{g_i\}_{i \in [12]}$ denote a carefully selected set of $7 \times 7$ minors of $G$ for which at least one of the variables $X_{j,j + 1}$ has degree 0 or 1.  Observe $P(f)\cap Z(g) = \emptyset$ implies $m(S) > 0.2410$, which then implies that $S$ is not optimal.  For most $S \in \mathcal{R}_2$, we will show that $P(f)\cap Z(g)$ is empty by producing an approximate Positivstellensatz certificate.  With $x = (\mu,X_{12},X_{34},X_{56},X_{78})$, we define
\begin{align*}
 C_{m}(f) &:= \bigg\{\sum_{I\subseteq[10]}s_I(x)\prod_{i\in I}f_i(x):s_I(x)\in\Sigma^2[x],~\deg(s_I) \leq m ~ \forall I\subseteq[10]\bigg\},\\
 I_{m}(g) &:=\bigg\{\sum_{j\in[12]}t_j(x)g_j(x):t_j(x)\in\mathbf{R}[x],~\deg(t_j) \leq m ~ \forall j\in[12]\bigg\}.
\end{align*}
By Stengle's Positivstellensatz, it suffices to produce $h_1 \in C_{m_1}(f)$ and $h_2 \in I_{m_2}(g)$ for which $h = 1 + h_1 + h_2$ satisfies \eqref{eq.bound on coefficients}.  We use a Julia-based implementation~~\cite{bezanson17,dunning17} of sum of squares programming to obtain numerical solutions $\hat{h}_1 \in C_{m_1}$ and $\hat{h}_2 \in I_{m_2}$ for which $\hat{h}_1 + \hat{h}_2 \approx -1$, that is, $(\hat{h}_1,\hat{h}_2)$ provides a numerical approximation to a putative certificate that $P(f)\cap Z(g)$ is empty.  We will promote $(\hat{h}_1,\hat{h}_2)$ to an honest certificate by carefully rounding.  We write each scalar $s_I$ for $\hat{h}_1$ as a sum of squares and, for each term being squared, round its coefficients to five decimal places.  We similarly round the coefficients for each scalar $t_j$ for $\hat{h}_2$ to five decimal places.  Let $h_1 \in C_{m_1}(f)$ and $h_2 \in I_{m_2}(f)$ denote the resulting rounded polynomials with rational coefficients.  As each of our five variables is less than $1/4$ in absolute value, we may use \eqref{eq.simpler bound} in place of \eqref{eq.bound on coefficients} so as to apply \cref{approxpsatz} and conclude that $P(f) \cap Z(g) = \emptyset$ whenever the largest coefficient of $1 + h_1 + h_2$ is at most $1/5$ in absolute value.

We apply this strategy to each $S \in \mathcal{R}_2$ with $m_1 = 2$.  On a first run, we take $m_2 = 0$ and successfully eliminate 545 members of $\mathcal{R}_2$ in roughly 5 hours.  On a second run, we take $m_2 = 1$ and eliminate another 13 members of $\mathcal{R}_2$ in roughly 8 minutes.  This leaves us with only two members of $\mathcal{R}_2$ that could be optimal, and we proceed to use CAD to show that both are indeed optimal.

For these CAD queries, we again impose the constraint $f_i \geq 0$ for all $i\in[10]$, but we found that requiring $g_i = 0$ for all $i\in[12]$ resulted in CAD computations that did not terminate in a reasonable amount of time.  We instead relaxed to only require $g_i = 0$ for a select few $i\in[12]$ that only depend on four of the five decision variables.  For both of the remaining $S \in \mathcal{R}_2$, the corresponding CAD query reports that the optimal Gram matrix $G$ is equivalent to $G_6$.  One of these computations takes roughly 18 minutes, while the other takes over three hours.
\end{proof}

\section{Discussion}

In this paper, we classified the optimal $(d+2)$-codes for $\mathbf{RP}^{d-1}$ for both $d\in\{5,6\}$.
The next open case in this direction is $d=8$.
Sloane's putatively optimal code~\cite{sloaneDatabase} is equiangular:
\[
\normalsize{
G_8:=
\scriptsize{\left[\begin{array}{rrrrrrrrrr}
1&-\mu&\mu&\mu&\mu&-\mu&\mu&\mu&-\mu&\mu\\
-\mu&1&-\mu&\mu&\mu&-\mu&\mu&\mu&\mu&\mu\\
\mu&-\mu&1&-\mu&\mu&-\mu&\mu&-\mu&\mu&-\mu\\
\mu&\mu&-\mu&1&\mu&-\mu&-\mu&-\mu&\mu&-\mu\\
\mu&\mu&\mu&\mu&1&\mu&-\mu&-\mu&-\mu&\mu\\
-\mu&-\mu&-\mu&-\mu&\mu&1&\mu&\mu&\mu&-\mu\\
\mu&\mu&\mu&-\mu&-\mu&\mu&1&-\mu&-\mu&-\mu\\
\mu&\mu&-\mu&-\mu&-\mu&\mu&-\mu&1&\mu&-\mu\\
-\mu&\mu&\mu&\mu&-\mu&\mu&-\mu&\mu&1&\mu\\
\mu&\mu&-\mu&-\mu&\mu&-\mu&-\mu&-\mu&\mu&1\\
\end{array}\right]},}
\]
where $\mu$ is given in \cref{table.opt packings}.

We expect that our current approach can already be used to partially tackle this case.
For example, our methods in Section~3 should be able to treat $\mathcal{R}_1$, but recall that it took 5 hours for us to rule out most of $\mathcal{R}_2$ in the $d=6$ case.
Considering $\mathcal{R}_2$ is over a thousand times larger in the $d=8$ case (see \cref{table.opt packings}), our methods should require the better part of a year to tackle this larger case.
For the record, our naive enumeration of the members of $\mathcal{R}_1$ is too slow for this case, but faster approaches are available, e.g., \cite{szollosi18}.
Still, $\mathcal{R}_2$ requires new ideas.
Is there a way to treat $\mathcal{R}_2$ in an analogous manner to our treatment of $\mathcal{R}_1$ in Section~3?
Previous work classified optimal codes for $S^2$ and for $\mathbf{RP}^2$ by leveraging spherical geometry and linear programming instead of Positivstellensatz~\cite{musin12,musin15,mixon19a}; perhaps an analogous approach is available here?
At the end of our approach, we use CAD to exactly optimize $G$ for any surviving $S\in\mathcal{R}_2$.
In the $d=8$ case, these CAD queries may not terminate in a reasonable amount of time.
We note that in the $d=6$ case, the Positivstellensatz step quickly produced an improved lower bound of $\mu_{8,6}\geq0.24$ before this CAD step, and including this information in our CAD query cut the three-hour runtime in half.
It might be possible to obtain improved lower bounds on $\mu_{10,8}$ even if CAD takes too long.

There might be some improvements available in our application of Positivstellensatz.
For example, we rounded our numerical approximations of Positivstellensatz certificates to five decimal places before using exact arithmetic to verify that the result satisfies the bound \eqref{eq.simpler bound}.
The exact arithmetic step might be faster if we had rounded to four decimal places (say), but we expect the bound \eqref{eq.simpler bound} to be violated if we round too much.
Next, in order for Positivstellensatz and CAD to have reasonable runtimes, we relaxed various determinant constraints.
While we have some heuristics for when a relaxation is good (e.g., some of the remaining polynomials have low degree in certain variables), this process remains an artform that deserves a careful treatment.

In prior work, numerical applications of Stengle's Positivstellensatz come in two different types.
The first type solves a sum-of-squares program numerically, and then performs what appears to be a handcrafted rounding step to ensure that $-1$ exactly resides in the set $C(f)+I(g)$; see~\cite{parrilo03}, for example.
This approach was not suitable for our purposes since we were solving hundreds of sum-of-squares programs.
The second type takes the numerical result that $h\approx-1$ resides in $C(f)+I(g)$ as sufficient evidence that $P(f)\cap Z(g)$ is empty; see~\cite{davis11}, for example.
Since this does not constitute a proof, it was also not suitable for our purposes.
Presumably, \cref{approxpsatz} could replace the ad-hoc strategy of the first type and give theoretical justification for the second type.
Furthermore, it would be interesting if \cref{approxpsatz} could provide sum-of-squares certificates of lower degree than Stengle's original Positivstellensatz.

Finally, we point out some problems that are adjacent to ours.
While we have focused on real projective spaces, the analogous question can be posed in complex projective spaces $\mathbf{CP}^{d-1}$.
Here, the optimal $n$-codes are known for $n\leq d+1$, but they are similarly mysterious for $n=d+2$.
Since $\mathbf{CP}^1$ is the $2$-sphere, the optimal $4$-code for $\mathbf{CP}^1$ is given by the vertices of the tetrahedron.
More generally, Bukh and Cox~\cite{bukh19} characterize the optimal $(d+2)$-codes for $\mathbf{CP}^{d-1}$ for every $d\equiv 2\bmod 4$.
These are the only solved cases.
For the $d=3$ case, Jasper, King and Mixon~\cite{jasper19} conjecture that the optimal $5$-code is given by the lines spanned by the columns of
\[
\left[\begin{array}{lllll}
a&b&b&c&c\\
b&a&b&cw&cw^2\\
b&b&a&cw^2&cw
\end{array}\right],
~a=\frac{\sqrt{13}+\sqrt{2+\sqrt{13}}-1}{3\sqrt{3}},
~b=\sqrt{\frac{1-a^2}{2}},
~c=\frac{1}{\sqrt{3}},
~w=e^{2\pi i/3}.
\]
Furthermore, King will buy a coffee for the first person to prove this conjecture~\cite{dustinBlog}.
Our methods do not easily transfer to this setting since sign patterns in the Gram matrix are no longer discrete.

The analogous question has also been posed in the sphere $S^{d-1}$, where the optimal $n$-codes are known for $n\leq 2d$.
For $n=2d+1$, little is known.
For $d=2$, the optimal code is given by five uniformly spaced points on the circle, and the $d=3$ case was solved by Sch\"{u}tte and van der Waerden~\cite{schutte51} in 1951.
Ballinger et al.~\cite{ballinger09} offer a conjecture that treats all dimensions simultaneously:
Let $S\in\mathbf{R}^{(d-1)\times d}$ be a matrix whose unit-norm columns form the vertices of a regular simplex.
The putatively optimal $(2d+1)$-code for $S^{d-1}$ is unique up to isometry and given by the columns of
\[
\left[\begin{array}{ccc}
1&\alpha&\beta\\
0&\sqrt{1-\alpha^2}\cdot S&-\sqrt{1-\beta^2}\cdot S
\end{array}\right],
\]
where $\alpha$ is the unique root between $0$ and $1/d$ of
\[
(d^3-4d^2+4d)x^3-d^2x^2-dx+1,
\]
and $\beta$ is the unique root between $-1$ and $1$ of
\[
\alpha x+\frac{1}{d-1}\sqrt{(1-\alpha^2)(1-x^2)}-\alpha.
\]
Our methods do not easily transfer to this setting since the contact graphs are far less dense, meaning the resulting programs have more decision variables.

\section*{Acknowledgments}

DGM was partially supported by AFOSR FA9550-18-1-0107, NSF DMS 1829955, and the 2019 Kalman Visiting Fellowship at the University of Auckland.  HP was partially supported by an AMS-Simons Travel Grant.

\bibliography{sosbib}

\providecommand{\bysame}{\leavevmode\hbox to3em{\hrulefill}\thinspace}
\providecommand{\MR}{\relax\ifhmode\unskip\space\fi MR }
\providecommand{\MRhref}[2]{%
  \href{http://www.ams.org/mathscinet-getitem?mr=#1}{#2}
}
\providecommand{\href}[2]{#2}
\begin{thebibliography}{10}

\bibitem{ballinger09}
Brandon Ballinger, Grigoriy Blekherman, Henry Cohn, Noah Giansiracusa,
  Elizabeth Kelly, and Achill Sch\"{u}rmann, \emph{Experimental study of
  energy-minimizing point configurations on spheres}, Experiment. Math.
  \textbf{18} (2009), no.~3, 257--283. \MR{2555698}

\bibitem{bandeira13}
Afonso~S. Bandeira, Matthew Fickus, Dustin~G. Mixon, and Percy Wong, \emph{The
  road to deterministic matrices with the restricted isometry property}, J.
  Fourier Anal. Appl. \textbf{19} (2013), no.~6, 1123--1149. \MR{3132908}

\bibitem{benedetto06}
John~J Benedetto and Joseph~D Kolesar, \emph{Geometric properties of
  {G}rassmannian frames for {$\mathbf{R}^2$} and {$\mathbf{R}^3$}}, EURASIP
  Journal on Advances in Signal Processing \textbf{2006} (2006), no.~1, 049850.

\bibitem{bezanson17}
Jeff Bezanson, Alan Edelman, Stefan Karpinski, and Viral~B. Shah, \emph{Julia:
  a fresh approach to numerical computing}, SIAM Rev. \textbf{59} (2017),
  no.~1, 65--98. \MR{3605826}

\bibitem{bukh19}
Boris Bukh and Christopher Cox, \emph{Nearly orthogonal vectors and small
  antipodal spherical codes}, Israel J. Math. (2019+), to appear.

\bibitem{bussemaker81}
F.~C. Bussemaker, R.~A. Mathon, and J.~J. Seidel, \emph{Tables of two-graphs},
  Combinatorics and graph theory ({C}alcutta, 1980), Lecture Notes in Math.,
  vol. 885, Springer, Berlin-New York, 1981, pp.~70--112. \MR{655610}

\bibitem{casselman04}
Bill Casselman, \emph{The difficulties of kissing in three dimensions}, Notices
  Amer. Math. Soc. \textbf{51} (2004), no.~8, 884--885. \MR{2145822}

\bibitem{chung89}
F.~R.~K. Chung, R.~L. Graham, and R.~M. Wilson, \emph{Quasi-random graphs},
  Combinatorica \textbf{9} (1989), no.~4, 345--362. \MR{1054011}

\bibitem{collins75}
George~E. Collins, \emph{Quantifier elimination for real closed fields by
  cylindrical algebraic decomposition}, Lecture Notes in Comput. Sci.
  \textbf{33} (1975), 134--183. \MR{0403962}

\bibitem{conway96}
John~H. Conway, Ronald~H. Hardin, and Neil J.~A. Sloane, \emph{Packing lines,
  planes, etc.: packings in {G}rassmannian spaces}, Experiment. Math.
  \textbf{5} (1996), no.~2, 139--159. \MR{1418961}

\bibitem{davenport88}
James~H. Davenport and Joos Heintz, \emph{Real quantifier elimination is doubly
  exponential}, J. Symbolic Comput. \textbf{5} (1988), no.~1-2, 29--35.
  \MR{949111}

\bibitem{davis11}
J.~M. {Davis} and G.~{Eisenbarth}, \emph{The positivstellensatz and
  nonexistence of common quadratic lyapunov functions}, 2011 IEEE 43rd
  Southeastern Symposium on System Theory, March 2011, pp.~55--58.

\bibitem{delaat15}
David de~Laat and Frank Vallentin, \emph{A semidefinite programming hierarchy
  for packing problems in discrete geometry}, Math. Program. \textbf{151}
  (2015), no.~2, Ser. B, 529--553. \MR{3348162}

\bibitem{dunning17}
Iain Dunning, Joey Huchette, and Miles Lubin, \emph{Jump: A modeling language
  for mathematical optimization}, SIAM Review \textbf{59} (2017), no.~2,
  295--320.

\bibitem{fejes65}
L.~Fejes~T\'{o}th, \emph{Distribution of points in the elliptic plane}, Acta
  Math. Acad. Sci. Hungar. \textbf{16} (1965), 437--440. \MR{184139}

\bibitem{fickus18}
Matthew Fickus, John Jasper, and Dustin~G. Mixon, \emph{Packings in real
  projective spaces}, SIAM J. Appl. Algebra Geom. \textbf{2} (2018), no.~3,
  377--409. \MR{3831238}

\bibitem{fickus15}
Matthew Fickus and Dustin~G. Mixon, \emph{Tables of the existence of
  equiangular tight frames}, arXiv preprint arXiv:1504.00253 (2015).

\bibitem{foucart13}
Simon Foucart and Holger Rauhut, \emph{A mathematical introduction to
  compressive sensing}, Applied and Numerical Harmonic Analysis,
  Birkh\"{a}user/Springer, New York, 2013. \MR{3100033}

\bibitem{golay49}
Marcel J.~E. Golay, \emph{Notes on digital coding}, Proc. I.R.E. \textbf{37}
  (1949), 657. \MR{4021352}

\bibitem{gross11}
David Gross, \emph{Recovering low-rank matrices from few coefficients in any
  basis}, IEEE Trans. Inform. Theory \textbf{57} (2011), no.~3, 1548--1566.
  \MR{2815834}

\bibitem{hamming50}
R.~W. Hamming, \emph{Error detecting and error correcting codes}, Bell System
  Tech. J. \textbf{29} (1950), 147--160. \MR{35935}

\bibitem{hanson19}
Brandon Hanson and Giorgis Petridis, \emph{Refined estimates concerning sumsets
  contained in the roots of unity}, arXiv preprint arXiv:1905.09134 (2019).

\bibitem{jasper19}
John Jasper, Emily~J. King, and Dustin~G. Mixon, \emph{{Game of Sloanes: best
  known packings in complex projective space}}, Wavelets and Sparsity XVIII
  (Dimitri Van~De Ville, Manos Papadakis, and Yue~M. Lu, eds.), vol. 11138,
  International Society for Optics and Photonics, SPIE, 2019, pp.~416 -- 425.

\bibitem{mallows75}
C.~L. Mallows and N.~J.~A. Sloane, \emph{Two-graphs, switching classes and
  {E}uler graphs are equal in number}, SIAM J. Appl. Math. \textbf{28} (1975),
  876--880. \MR{427128}

\bibitem{mishra93}
Bhubaneswar Mishra, \emph{Algorithmic algebra}, Texts and Monographs in
  Computer Science, Springer-Verlag, New York, 1993. \MR{1239443}

\bibitem{dustinBlog}
Dustin~G. Mixon, \emph{Game of sloanes}, \textit{Short, Fat Matrices},
  \url{https://dustingmixon.wordpress.com/2019/08/20/game-of-sloanes/}.

\bibitem{mixon19a}
Dustin~G. Mixon and Hans Parshall, \emph{The optimal packing of eight points in
  the real projective plane}, Exp. Math. (2019+), to appear.

\bibitem{mixon13}
Dustin~G. Mixon, Christopher~J. Quinn, Negar Kiyavash, and Matthew Fickus,
  \emph{Fingerprinting with equiangular tight frames}, IEEE Trans. Inform.
  Theory \textbf{59} (2013), no.~3, 1855--1865. \MR{3030758}

\bibitem{musin12}
Oleg~R. Musin and Alexey~S. Tarasov, \emph{The strong thirteen spheres
  problem}, Discrete Comput. Geom. \textbf{48} (2012), no.~1, 128--141.
  \MR{2917205}

\bibitem{musin15}
\bysame, \emph{The {T}ammes problem for {$N=14$}}, Exp. Math. \textbf{24}
  (2015), no.~4, 460--468. \MR{3383477}

\bibitem{parrilo03}
Pablo~A. Parrilo and Bernd Sturmfels, \emph{Minimizing polynomial functions},
  Algorithmic and quantitative real algebraic geometry ({P}iscataway, {NJ},
  2001), DIMACS Ser. Discrete Math. Theoret. Comput. Sci., vol.~60, Amer. Math.
  Soc., Providence, RI, 2003, pp.~83--99. \MR{1995016}

\bibitem{renes04}
Joseph~M. Renes, Robin Blume-Kohout, A.~J. Scott, and Carlton~M. Caves,
  \emph{Symmetric informationally complete quantum measurements}, J. Math.
  Phys. \textbf{45} (2004), no.~6, 2171--2180. \MR{2059685}

\bibitem{schutte51}
K.~Sch\"{u}tte and B.~L. {van der Waerden}, \emph{Auf welcher {K}ugel haben
  {$5$}, {$6$}, {$7$}, {$8$} oder {$9$} {P}unkte mit {M}indestabstand {E}ins
  {P}latz?}, Math. Ann. \textbf{123} (1951), 96--124. \MR{0042150}

\bibitem{schutte53}
K.~Sch\"{u}tte and B.~L. van~der Waerden, \emph{Das {P}roblem der dreizehn
  {K}ugeln}, Math. Ann. \textbf{125} (1953), 325--334. \MR{53537}

\bibitem{seidel76}
J.~J. Seidel, \emph{A survey of two-graphs}, Colloquio {I}nternazionale sulle
  {T}eorie {C}ombinatorie ({R}ome, 1973), {T}omo {I}, 1976, pp.~481--511. Atti
  dei Convegni Lincei, No. 17. \MR{0550136}

\bibitem{shannon48}
C.~E. Shannon, \emph{A mathematical theory of communication}, Bell System Tech.
  J. \textbf{27} (1948), 379--423, 623--656. \MR{26286}

\bibitem{sloaneDatabase}
Neil J.~A. Sloane, \emph{Packings in grassmannian spaces},
  \url{http://neilsloane.com/grass/}.

\bibitem{stengle74}
Gilbert Stengle, \emph{A nullstellensatz and a positivstellensatz in
  semialgebraic geometry}, Math. Ann. \textbf{207} (1974), 87--97. \MR{332747}

\bibitem{strohmer03}
Thomas Strohmer and Robert~W. Heath, Jr., \emph{Grassmannian frames with
  applications to coding and communication}, Appl. Comput. Harmon. Anal.
  \textbf{14} (2003), no.~3, 257--275. \MR{1984549}

\bibitem{szollosi18}
Ferenc Sz\"{o}ll\H{o}si and Patric R.~J. \"{O}sterg\aa rd, \emph{Enumeration of
  {S}eidel matrices}, European J. Combin. \textbf{69} (2018), 169--184.
  \MR{3738150}

\bibitem{tammes30}
Pieter Merkus~Lambertus Tammes, \emph{On the origin of number and arrangement
  of the places of exit on the surface of pollen-grains}, Recueil des travaux
  botaniques n{\'e}erlandais \textbf{27} (1930), no.~1, 1--84.

\end{thebibliography}
\bibliographystyle{amsplain}

\end{document}